\RequirePackage{fix-cm}
\documentclass[smallextended]{svjour3}       
\smartqed  
\usepackage{stmaryrd}
\usepackage{graphicx, amssymb, amsmath, amsbsy}
\def\C{\mathbb C}
\def\R{\mathbb R}

\def\N{\mathbb N}

\usepackage[norefs, nocites]{refcheck}
\usepackage[breakall, fit]{truncate}
\usepackage[bordercolor=white, color=white, textwidth=100pt]{todonotes}

\usepackage{mathtools}
\usepackage[ocgcolorlinks, linkcolor=blue, citecolor=blue]{hyperref}

\DeclarePairedDelimiter\norm{\lVert}{\rVert}
\newcommand{\jumpF}[1]{\llbracket #1 \rrbracket_F}
\newcommand{\triple}[1]{|\!|\!|#1|\!|\!|_S}

\begin{document}

\title{{Optimal approximation of unique continuation}
\thanks{
E.B.: supported by the EPSRC grants EP/T033126/1 and EP/V050400/1.\\
M. N.: this work was supported by the project "The Development of Advanced and Applicative Research Competencies in the Logic of STEAM + Health” /POCU/993/6/13/153310, project co-financed by the European Social Fund through The Romanian Operational Programme Human Capital 2014-2020.\\
L. O.: Co-funded by the European Union (ERC, LoCal, 101086697). Views and opinions expressed are however those of the authors only and do not necessarily reflect those of the European Union or the European Research Council. Neither the European Union nor the granting authority can be held responsible for them.
Co-funded by the Research Council of Finland (347715, 3530969).
}
}


\author{Erik Burman$^{1}$ \and Mihai Nechita$^{2,3}$ \and Lauri Oksanen$^{4}$}

\authorrunning{Erik Burman et al} 

\institute{E. Burman \at
              $^{1}$Department of Mathematics, University College London, London WC1E 6BT, UK\\
              \email{e.burman@ucl.ac.uk}           
           \and
           M. Nechita \at
              $^{2}$Tiberiu Popoviciu Institute of Numerical Analysis, Romanian Academy\\
              $^{3}$Department of Mathematics, Babeș-Bolyai University, Cluj-Napoca, Romania\\
              \email{mihai.nechita@ictp.acad.ro}
           \and 
           L. Oksanen \at
           $^{4}$Department of Mathematics and Statistics, University of Helsinki,
           P.O. 68, 00014 University of Helsinki, Finland\\
           \email{lauri.oksanen@helsinki.fi}
}

\date{Received: date / Accepted: date}

\maketitle

\begin{abstract}
We consider numerical approximations of ill-posed elliptic problems with conditional stability.
The notion of {\emph{optimal error estimates}} is defined including both convergence with respect to discretization and perturbations in data.
The rate of convergence is determined by the conditional stability of the underlying continuous problem and the polynomial order of the approximation space.
A proof is given that no approximation can converge at a better rate than that given by the definition without increasing the sensitivity to perturbations, thus justifying the concept.
A recently introduced class of primal-dual finite element methods with weakly consistent regularisation is recalled and the associated error estimates are shown to be optimal in the sense of this definition. 
\end{abstract}

\section{Introduction}
\label{intro}
Arguably one of the most fundamental results in finite element analysis is the best approximation result for the Galerkin method, known as Cea's Lemma \cite{Cea64}, which together with approximation estimates for finite element functions results in quasi-optimal error estimates for finite element methods \cite{Zlamal68}, \cite{Nit70}, \cite{Bab70}, \cite{Hel71}.
This result, that we will review below, essentially says that if a $(2m)$-order elliptic problem, $m\ge 1$, is approximated with $H^m$-conforming finite elements of local polynomial order $p$ the error in $H^m$-norm is of the order $h^{p+1-m}$ for a sufficiently smooth solution and that this rate is optimal compared to approximation: the best interpolant of the exact solution has similar accuracy.

For ill-posed elliptic problems the situation is different.
On the continuous level existence can only be guaranteed after regularisation of the problem.
The two main approaches are Tikhonov regularisation \cite{TA77} and quasi-reversibility \cite{LL69}.
These two approaches are strongly related (see for instance \cite{BR18}).
The main effort in the error analysis has been to estimate the perturbation induced by the addition of regularisation, and how to choose the associated regularisation operator or parameter \cite{Miller73}, \cite{Nat84}, \cite{Lu88}, \cite{Bour05}, \cite{IJ15}.
The error due to approximation in finite dimensional spaces of such regularised problems has also been analysed \cite{Nat77}, \cite{EHN88}, \cite{MP01}.

There is also a rich literature on projection methods for ill-posed problems where the discretisation serves as regularisation and refinement has to stop as soon as the effect of perturbations in data becomes dominant \cite{Natt77}, \cite{Engl83}, \cite{EN87}, \cite{HAG02}, \cite{Kalt00}.
These methods are often based on least squares methods and the convergence of the approximate solution to the exact solution for unperturbed data has been proven in several works.
There are also different stopping criteria for mesh refinement in order to avoid degeneration due to pollution from perturbations.
However no results on rates of convergence where the discretisation errors and the perturbation errors are both included appear in these references.

The use of \emph{conditional stability} (continuous dependence on data under the assumption of a certain a priori bound) to obtain more complete error estimates has been proposed in \cite{Bu13}, \cite{Bu14b}, \cite{Bu16}, \cite{BO18} for a class of finite element methods based on weakly consistent regularisation/stabilisation in a primal-dual framework.
Here stability is obtained through a combination of consistent stabilisation and Tikhonov regularisation, scaled with the mesh parameter to obtain weak consistency.
The upshot is that for this class of methods an error analysis exists, where the computational error is bounded in terms of the mesh parameter and perturbations of data, with constants depending on Sobolev norms of the exact solution.
Similarly to the well-posed case, the error estimates for this approach combine the stability of the physical problem with the numerical stability of the computational method and the approximability of the finite element space.
Contrary to the well-posed case, numerical stability can not be deduced from the physical stability, but has to be a consequence of the design of the stabilisation terms.
This means that the stabilisation in this framework is bespoke, and must be designed to combine optimal (weak) consistency and sufficient numerical stability.
There is often a tension between these two design criteria. As noted above, sometimes Tikhonov regularisation, scaled with the mesh parameter, may be used in the framework.
An interesting feature is that the bespoke character also allows for the integration of the dependence of the estimates on physical parameters and different problems regimes \cite{BNO19}, \cite{BNO20a}.
Other physical models that have been considered in this framework include data assimilation for fluids
\cite{BO18},\cite{BH18},\cite{BBFV20}, or wave equations \cite{BFO20a}.
Common for all these references is the fact that the error estimates reflect the stability of the continuous problem and the approximation order of the finite element space, which seems to be an optimality property of the methods. No rigorous proof, however, has been given for this optimality.
The objective of the present work is to show, in the model case of unique continuation for Laplace equation, that the proposed error estimates are indeed optimal.

For ill-posed PDEs that are conditionally stable, error estimates in terms of the modulus of continuity in the conditional stability, the consistency error and the best approximation error have also been obtained in \cite{DMS23}.
Based on least squares with the norms and the regularisation term dictated by the conditional stability estimate, this variation of quasi-reversibility relies on working with discrete dual norms and constructing Fortin projectors.
By choosing the regularisation parameter in terms of the consistency error and the best approximation error, the obtained error bound reflects the conditional stability estimate (qualitatively optimal).
Conditional stability estimates have also been used to obtain some bounds on the generalization error for physics-informed neural networks solving ill-posed PDEs \cite{MM22}. 
The question of optimality for both these kind of methods is included in our discussion.

Another well-known ill-posed problem is analytic continuation, which, similarly to unique continuation, possesses conditional stability  under the assumption of an a priori bound.
We will not discuss this problem here; for its conditional stability/conditioning and numerical approximations, we refer the reader to \cite{Trefethen20}, \cite{Trefethen23} and the references therein.

\subsection{Unique continuation problem}

Let $0 < r_1 < r_2 < R$ and write $\omega := B(r_1)$ and $B := B(r_2)$ where $B(r)$ is the open ball of radius $r > 0$, with the centre at the origin in $\R^n$.
The objective is to solve the continuation problem: given the restriction $u|_\omega$ to the subset $\omega$, find the restriction $u|_B$ when $u$ satisfies $\Delta u = 0$ in $B(R)$.

Further, letting $r_2 < r_3 < R$ and writing $\Omega := B(r_3)$, it is classical \cite{Fritz60} that the following conditional stability estimate holds:
\begin{equation}\label{eq:3sphere}
\|u\|_{L^2(B)} \lesssim \|u\|_{L^2(\omega)}^\alpha \|u\|_{L^2(\Omega)}^{(1-\alpha)},
\end{equation}
where $\alpha \in (0,1)$ and the implicit constant do not depend on the harmonic function $u$.
Estimate \eqref{eq:3sphere} is often called a three-ball inequality and in this case the constants can be given explicitly, see Theorem \ref{thm:opt_alpha} below.
We may view the unique continuation problem as finding $u \in H^1(\Omega)$ such that
\begin{equation}\label{eq:UC}
\left\{\begin{array}{rcl}
-\Delta u &=& 0 \mbox{ in } \Omega, \\
u\vert_\omega &=& q \mbox{ in } \omega,
\end{array} \right.
\end{equation}
with a priori knowledge on the size of the solution in $\Omega$ as prescribed by the $L^2(\Omega)$-norm in \eqref{eq:3sphere}.  

\subsection{Motivation and outline}
The motivation of this paper comes from error estimates obtained for primal-dual finite element methods applied to this problem, with perturbed data $u\vert_\omega =q + \delta q$, of the form
\begin{equation}\label{eq:canonical_estimate}
\|u - u_h\|_{L^2(B)} \lesssim h^\alpha \|u\|_{H^2(\Omega)} + h^{\alpha - 1} \|\delta q\|_{L^2(\omega)},
\end{equation}
which have been shown in \cite{BHL18}, \cite{BNO19}, \cite{BNO20a} in different variations of the second order elliptic equation.
Here $\alpha \in (0,1)$ is the exponent in \eqref{eq:3sphere} and $h>0$ denotes the mesh parameter defining the characteristic length scale of the finite dimensional space.
This is in the case of piecewise affine approximation, however the estimate generalises in a natural way to higher order approximation, as we shall see below.
One can obtain a similar bound in the $H^1$-norm over $B$. 
In the counterfactual case that $\alpha = 1$ one would then recover an estimate that is optimal compared to interpolation.
Hence a natural question is if the bound \eqref{eq:canonical_estimate} in the $L^2$-norm can be improved upon, since it is suboptimal with one order in $h$ when compared to interpolation in the case $\alpha =1$.

We show in this paper that if the coefficient $\alpha$ in \eqref{eq:3sphere} is optimal and depends continuously on $r_3$ (Theorem \ref{thm:opt_alpha}),
then regardless of the underlying method, no sequence of approximations to \eqref{eq:UC} can converge with a rate better than that given by \eqref{eq:canonical_estimate} (Theorem \ref{thm:optimal}) without increasing the sensitivity to perturbations.
We also point out that although the discussion focuses on the finite element method, the definition of optimal convergence given and the proof of optimality hold for any method producing an approximating sequence in $H^1$ (or relying on such a sequence for the analysis).

The paper is organised as follows.
In Section \ref{sec:rev_FEM} we will discuss the notion of optimality of finite element approximations.
First we revisit the classical finite element analysis for well-posed problems.
In Section \ref{sec:ill_posed_opt} we then discuss how the ideas of the well-posed case translate to the ill-posed case.
This leads us to a definition of optimal approximation for the problem \eqref{eq:UC} and we prove in Section \ref{sec:optimal} that no approximation method can converge in a better way than that given by this definition.
Finally in Section \ref{sec:FEM} we show that optimality can indeed be attained by presenting a finite element method with optimal error estimates which extend \eqref{eq:canonical_estimate} for higher order approximations.

\section{Optimal error estimates for elliptic problems}\label{sec:rev_FEM}
In this section we will first briefly recall the theory of optimal error estimates for Galerkin approximations of well-posed second order elliptic problems.
We then consider the ill-posed model problem \eqref{eq:UC} and discuss how the construction that led to optimal approximation in the well-posed case can be adapted to this situation.
This leads us to a definition of optimality of approximate solutions in the ill-posed case. We let $V:=H^1(\Omega)$ and $V_0:=H^1_0(\Omega)$.

For simplicity we consider the Poisson problem for $f \in V_0'$:
\begin{equation}\label{eq:Poisson}
\left\{\begin{array}{rcl}
-\Delta u & = & f \mbox{ in } \Omega \\
u & = & 0 \mbox{ on } \partial \Omega.
\end{array} \right.
\end{equation}
We define the associated weak formulation by: find $u \in V_0$ such that
\begin{equation}\label{eq:weak}
a(u,v) = \ell(v) \quad \forall v \in V_0,
\end{equation}
where $a(u,v):= \int_\Omega \nabla u \cdot \nabla v ~\mbox{d}x$ and $\ell(v) := \langle f,v\rangle_{V'_0,V_0}$.
It is well known that $a$ is a coercive ($V$-elliptic), continuous bilinear form on $V_0 \times V_0$, and $\ell$ is bounded and continuous on $V_0$. 

It then follows from Lax-Milgram's lemma \cite{LM54} that the weak formulation admits a unique solution satisfying the stability estimate
\begin{equation}\label{eq:stab}
\|u\|_V \leq \|\ell\|_{{V_0'}}
\end{equation}
where $\|\cdot\|_V := \sqrt{a(\cdot,\cdot)}$ and the dual norm is defined by
\[
\|\ell\|_{{V_0'}}:= \sup_{v \in V_0 \setminus \{0\}} \frac{|\ell(v)|}{\|v\|_V}.
\]
Introducing a finite dimensional subspace $V_{0h} \subset V_0$ we define the Galerkin method, find $u_h \in V_{0h}$ such that
\begin{equation}\label{eq:Galerkin}
a(u_h,v_h) = \ell_\delta(v_h) \quad \forall v_h \in V_{0h},
\end{equation}
where $\ell_\delta$ denotes a perturbed right hand side
$\ell_\delta(v) := \langle f + \delta f,v\rangle_{V_0',V_0}$,
with $\delta f \in {V_0'}$ and we assume $\|\delta f\|_{{V_0'}}$ to be known.
The associated linear system is invertible since $a$ is coercive.

Let $\bar{u}_h\in V_{0 h}$ be the solution for the unperturbed right hand side, satisfying
$$a\left(\bar{u}_h, v_h\right)=\ell\left(v_h\right)\quad \forall v_h \in V_{0 h}.$$
Then we have that $\left\|u-\bar{u}_h\right\|_V=\inf _{w_h \in V_{0 h}}\left\|u-w_h\right\|_V$, by Galerkin orthogonality,
and that $\left\|u_h-\bar{u}_h\right\|_V \leq\|\delta f\|_{V_0^{\prime}}$.
Since $\Delta: V_0 \rightarrow V_0^{\prime}$ is an isomorphic isometry, an application of the triangle-inequality for the approximation error $e := u- u_h$ gives that
\begin{equation}\label{eq:error_bound}
	\|e\|_V \leq \inf_{w_h \in V_{0h}} \|\Delta (u - w_h)\|_{{V_0'}}+
	\|\delta f\|_{{V_0'}}.
\end{equation}
This is equivalent to the classical result of Cea's lemma, but written in a form suitable for our purposes.
If $u\in H^{k+1}(\Omega)$ and $V_{0h}$ is the space of $H^1$-conforming piecewise polynomial finite elements of order $k$ we immediately have by approximation that
\begin{equation}\label{eq:error_bound_highorder}
	\|e\|_V \lesssim h^k |u|_{H^{k+1}(\Omega)} + \|\delta f\|_{{V_0'}},
\end{equation}
where $|\cdot|_{H^{k+1}(\Omega)}$ stands for the seminorm.
Observe how the Lipschitz stability of \eqref{eq:stab} combines with the approximation properties of the finite element space to yield an optimal error estimate.
Perturbations in data lead to stagnation of the error at the level of the perturbation.
\subsection{Optimal three-ball estimate}
Three-ball estimates such as \eqref{eq:3sphere} for solutions of second-order elliptic equations are well-known in the literature, see e.g. the review \cite{ARRV09} or \cite{Bru95}.
However, such results typically contain constants that depend implicitly on the geometry and the coefficients of the differential operator, and whose optimality is not clear \cite{Bru95}.
We aim here to give a result in the case of the Laplace operator which, barring optimality, is a variation of existing results in the literature, see \cite[Theorem 1]{Miller62} and \cite[Eq. (1.2)]{kuusi2021}.
We will consider only the two and three dimensional cases, for which we prove the following three-ball estimate in $L^2$-norms with optimal explicit constants.
\begin{theorem}\label{thm:opt_alpha}
Let $n\in\{2,3\}$ and $B(r) \subset \R^n$ be the open ball of radius $r > 0$.
Let $0 < r_1 < r_2 < r_3$. 
Then for all harmonic functions $u$ there holds
\begin{equation}\label{eq:3sphere_sharp}
\|u\|_{L^2(B(r_2))} \le \|u\|_{L^2(B(r_1))}^\alpha \|u\|_{L^2(B(r_3))}^{(1-\alpha)},
\end{equation}
where
    \begin{align}\label{optimal_alpha}
\alpha := \frac{\beta}{1+\beta} = \frac{\log(r_3) - \log(r_2)}{\log(r_3) - \log(r_1)}, 
\quad 
\beta := \frac{\log(r_3) - \log(r_2)}{\log(r_2) - \log(r_1)}.
    \end{align}
Moreover, there does not exist $\tilde \alpha > \alpha$ such that 
\begin{equation}\label{eq:3sphere_toogood}
\|u\|_{L^2(B(r_2))} \lesssim \|u\|_{L^2(B(r_1))}^{\tilde \alpha} \|u\|_{L^2(B(r_3))}^{(1-\tilde \alpha)}.
\end{equation}
\end{theorem}
\begin{proof}
For any $r > 0$ and $\theta \in (0,1)$ there holds 
\begin{align} 
\label{basic_threeballs}
\left\| u \right\|_{L^2(B_{r} )} 
\leq 
\left\| u \right\|_{L^2(B_{\theta r})}^{1/2} 
\left\| u \right\|_{L^2(B_{\theta^{-1} r} )}^{1/2}, 
\end{align}
see e.g. \cite[Eq. (1.2)]{kuusi2021}. We aim to transform this estimate into \eqref{eq:3sphere_sharp}. 
We take the logarithm of (\ref{basic_threeballs}) and write $f(t) = \log \left\| u \right\|_{L^2(B(\exp(t)) )}$ to obtain 
\begin{align*} 
f(\log(r)) \leq \frac12 f(\log(\theta r)) + \frac12 f(\log(\theta^{-1} r)).
\end{align*}
Notice that $\log(\theta r) + \log(\theta^{-1} r) = 2 \log (r)$, so that writing $t = \log(\theta r)$ and $s = \log(\theta^{-1} r)$, we obtain that 
\begin{align}\label{f_convex}
f(\tfrac12 (s + t )) \leq \tfrac12 ( f(s) + f(t)) ,
\end{align}
yielding convexity of $f$. Hence, for every $\alpha \in (0,1)$ and $s,t \in \R$  
\begin{align}\label{f_convex2}
f(\alpha s + (1-\alpha) t ) \leq  \alpha f(s) + (1-\alpha) f(t).
\end{align}

We now set $r = r_2$ and $\theta = r_1 / r_2$. Then $r_1 = \theta r$ and $r_3 = \theta^{-\beta} r$, with $\beta$ given by \eqref{optimal_alpha}.
Taking $s = \log (\theta r)$ and $t = \log (\theta^{-\beta} r)$ there holds
\begin{align*} 
\alpha s + (1-\alpha) t 
& =  
\alpha  \log \theta + \alpha  \log r  
- 
(1-\alpha) \beta \log \theta  + (1-\alpha) \log r
\\ & 
=
(\alpha - (1-\alpha) \beta ) \log \theta + \log r = \log r,
\end{align*}
since $(1-\alpha) \beta = \alpha$. With this choice, taking the exponential of \eqref{f_convex2} gives \eqref{eq:3sphere_sharp}.

Suppose now that \eqref{eq:3sphere_toogood} holds for some $\tilde \alpha > 0$.
We will show that $\tilde \alpha \le \alpha$.
Let us consider first the two dimensional case. 
Identifying $\R^2$ with $\C$, consider the function $u(z) = z^{n-1}$ which is harmonic for $n \in \N$.
The following argument is similarly valid for its real part. 
Using polar coordinates we have that, for $\rho>0$,
\begin{align}\label{harmonic_pow}
\norm{u}_{L^2(B(\rho))}^2 = 2 \pi \int_0^\rho r^{2(n-1) + 1} ~\mbox{d}r = 
c_n \rho^{2n}, \quad c_n = \frac{\pi}{n}.
\end{align}
Notice that we have equality in \eqref{eq:3sphere_sharp} for $\alpha = \beta/(1+\beta) = \frac{\log(r_3) - \log(r_2)}{\log(r_3) - \log(r_1)}$ .

Recalling that $r_1/r_2 = \theta,\, r_3/r_2 = \theta^{-\beta}$, estimate \eqref{eq:3sphere_toogood} reads as
\begin{align}
1 \lesssim \theta^{n (\tilde \alpha - \beta (1-\tilde \alpha))} .
\end{align}
As $n \in \N$ is arbitrary and $\theta \in (0,1)$ we must have 
$\tilde \alpha - \beta (1-\tilde \alpha) \le 0$.
In other words,
\begin{align*}
\tilde \alpha \le \frac{\beta}{1+\beta} = \alpha. 
\end{align*}

We turn to the three dimensional case, and consider the function 
\begin{align}\label{u_3d}
u(x^1, x^2, x^3) = z^{n-1}, \quad z = x^1 + i x^2.
\end{align}
As above, this is harmonic for $n \in \N$.
Passing to spherical coordinates there holds
\begin{align}\label{harmonic_pow_3d}
\norm{u}_{L^2(B(\rho))}^2 
= 2 \pi \int_0^\pi \int_0^\rho (r\sin \theta)^{2(n-1) + 1} r ~\mbox{d}r \mbox{d}\theta = 
c_n \rho^{2n + 1}, 
\end{align}
where the constant $c_n$ can be written using the Gamma function
\begin{align*}
c_n =  
\frac{2 \pi^{3/2} \Gamma(n)}{(2n + 1)\Gamma(n + 1/2)}.    
\end{align*}
The conclusion follows as in the two dimensional case.
\qed
\end{proof}
Note that the same explicit constants as in Theorem \ref{thm:opt_alpha} appear in Hadamard's three-circle theorem (in $L^{\infty}$-norms) for holomorphic functions.
\begin{remark}
In Theorem \ref{thm:opt_alpha} we proved the optimality and continuous dependence of the exponent $\alpha$ for unique continuation subject to the Laplace equation.
In a more general setting, a discussion of optimality of three-ball inequalities can be found in \cite{EFV06} for elliptic and parabolic problems.
In \cite{LNW10,LUW10} some cases in fluid mechanics and elasticity are considered for which optimality is claimed. 
\end{remark}

\subsection{Definition of optimal convergence for ill-posed problems with conditional stability}\label{sec:ill_posed_opt}
In this section we will try to mimic the development in the well-posed case for the problem \eqref{eq:UC} and point out where things go wrong.
We will do this with minimal reference to a particular approximation method to keep the discussion general.
However, in Section \ref{sec:FEM} we introduce a method for which the programme can be carried out.

First we will derive a weak formulation.
This time, since no boundary conditions are set on $u$, we must consider the trial space $V$.
To make form $a$ consistent with the problem, the test space must be chosen to be $V_0$, as keeping $V$ would imply a homogeneous Neumann condition on the boundary.
We may then write a weak formulation of the problem \eqref{eq:UC}, find $u \in V$ such that
$u\vert_{\omega} = q$ and 
\[
a(u,v) = 0 \quad \forall v \in V_0.
\]
We know that the exact solution satisfies this formulation and that \eqref{eq:3sphere} holds.
Assume now that we have an approximation $u_h \in V_h$ obtained using
the perturbed data $\tilde q  := q + \delta q$, where $\delta q \in L^2(\omega)$.
Observe that although for this data, most likely, no exact solution will exist, a discrete approximation of the unperturbed exact solution $u$ may still be constructed.
Similarly as before the error $e:=u-u_h$ satisfies
\[
a(e,v) = \ell_h(v) \quad \forall v \in V_0,
\]
where 
\[
\ell_h(v) := -a(u_h, v), \, \mbox{ with } \|\ell_h\|_{{V_0'}} = \|\Delta u_h\|_{{V_0'}}.
\]
Observe that even if $u_h$ is produced using a Galerkin procedure we can not use here the same techniques as when proving \eqref{eq:error_bound}, since the trial space $V$ in this case is bigger than the test space $V_0$.

As before we would now like to apply a stability estimate, this time \eqref{eq:3sphere}, using the right hand side on the perturbation. However this is not possible, since there is no right hand side in \eqref{eq:UC} and \eqref{eq:3sphere}.
Instead we first decompose $e = e_0 + \tilde e$, where $e_0 \in V_0$ solves the well-posed problem
\[
a(e_0,v) = \ell_h(v) \quad \forall v \in V_0,
\]
and $\tilde e$ solves \eqref{eq:UC} with $\tilde e\vert_{\omega} = (e - e_0)\vert_{\omega}$.
Using the triangle inequality and then applying \eqref{eq:stab} to $e_0$ and \eqref{eq:3sphere} to $\tilde e$ we arrive at
\[
\|e\|_{L^2(B)} \leq \|e_0\|_{V} + \|\tilde e\|_{L^2(B)} \lesssim \|\Delta
u_h\|_{{V_0'}} +\|\tilde e\|^{\alpha}_{L^2(\omega)} \|\tilde e\|^{1-\alpha}_{L^2(\Omega)}.
\]
Using once again the triangle inequality this leads to
\[
\|\tilde e\|^{\alpha}_{L^2(\omega)} \|\tilde
e\|^{1-\alpha}_{L^2(\Omega)} \lesssim (\|e\|_{L^2(\omega)} + \|\Delta
u_h\|_{{V_0'}})^{\alpha} (\|e\|_{L^2(\Omega)} + \|\Delta
u_h\|_{{V_0'}})^{1-\alpha}.
\]
We conclude that any approximation $u_h$ must satisfy the bound
\begin{equation}
	\label{eq:protobound}
	\begin{aligned}
		\|e\|_{L^2(B)} &\lesssim \|\Delta
		u_h\|_{{V_0'}} +(\|e\|_{L^2(\omega)} + \|\Delta
		u_h\|_{{V_0'}})^{\alpha} (\|e\|_{L^2(\Omega)} + \|\Delta
		u_h\|_{{V_0'}})^{1-\alpha} \\
		&\lesssim (\|q- u_h\|_{L^2(\omega)} + \|\Delta
		u_h\|_{{V_0'}})^{\alpha} (\|e\|_{L^2(\Omega)} + \|\Delta
		u_h\|_{{V_0'}})^{1-\alpha}.
	\end{aligned}
\end{equation}
If we assume that the term $\|e\|_{L^2(\Omega)}$ is bounded, then inequality \eqref{eq:protobound} gives an a posteriori bound for the
error on $B$ in the $L^2(B)$-norm.

For the sake of discussion, we will, for a moment, consider an approximation $u_h$ satisfying certain properties.
These properties can be thought of as design criteria for the numerical method, since as it turns out they lead to optimal convergence.
In Section \ref{sec:FEM} we construct a finite element method with these properties.

\begin{enumerate}
\item Bound on the equation residual:
\begin{equation}\label{eq:res}
\|\Delta u_h\|_{{V_0'}} \lesssim  C h^k|u|_{H^{k+1}(\Omega)} + \|\delta
q\|_{L^2(\omega)}.
\end{equation}
Observe that this means that the residual convergence in the ill-posed case is as good as the residual convergence in the well-posed case, see \eqref{eq:error_bound_highorder}.
\item Bound on the data fitting term:
\begin{equation}\label{eq:data}
\|q- u_h\|_{L^2(\omega)} \lesssim C h^k|u|_{H^{k+1}(\Omega)}+ \|\delta
q\|_{L^2(\omega)}
\end{equation}
This term is suboptimal by one order in $h$ compared to interpolation, but nothing can be gained by assuming better convergence since the
term always is dominated by the contribution from $\|\Delta u_h\|_{{V_0'}}$ in the bound \eqref{eq:protobound}.
Strengthening the norm on $\omega$ on the other hand is possible provided the perturbation $\delta q$ also has additional smoothness.
\item Finally we need to assume an a priori bound on $u_h$:
\begin{equation}\label{eq:apriori}
\|u_h\|_{L^2(\Omega)} \lesssim |u|_{H^{k+1}(\Omega)} +h^{-k} \|\delta q\|_{L^2(\omega)}.
\end{equation}
The rationale for this choice is that it is the strongest control that can be achieved through Tikhonov regularisation without affecting the convergence order when the data is unperturbed, assuming that the previous two assumptions hold.
\end{enumerate}
Injecting these three bounds in \eqref{eq:protobound} we get the error estimate
\begin{equation}\label{eq:optimal_conv}
\|u-u_h\|_{L^2(B)} \lesssim h^{\alpha k} \|u\|_{H^{k+1}(\Omega)}
+ h^{-(1-\alpha) k} \|\delta q\|_{L^2(\omega)}.
\end{equation}

A generic version of this estimate can be obtained by decoupling the rate of convergence from the sensitivity to perturbations, by considering the following bound
\begin{align}\label{eq:optimal_def}
	\norm{u-u_h}_{L^2(B)} 
	\lesssim 
	h^{\alpha_1 k} \norm{u}_{H^{k+1}(\Omega)} 
	+ h^{(\alpha_2 - 1) k} \norm{\delta q}_{L^2(\omega)},
\end{align}
for $\alpha_1,\, \alpha_2\in(0,1)$.
Denoting the upper bound here by $E(h):=h^{\alpha_1 k} \norm{u}_{H^{k+1}(\Omega)} + h^{(\alpha_2 - 1) k} \norm{\delta q}_{L^2(\omega)}$, 
we see that $E$ has a unique critical point $h_{min}$ on $(0,\infty)$, which is a minimum since $E''(h_{min}) > 0$. Hence we get that
$$
\norm{u-u_{h_{min}}}_{L^2(B)}
\lesssim
\|u\|_{H^{k+1}(\Omega)}^{1-\tilde{\alpha}} \|\delta q\|_{L^2(\omega)}^{\tilde{\alpha}},
$$
with $\tilde \alpha := \frac{\alpha_1}{1+\alpha_1 - \alpha_2}$.
Notice that $\tilde{\alpha} = \alpha$ when $\alpha_1 = \alpha_2 = \alpha$.
Considering convergence with respect to perturbations $\delta q$ in this bound, one would like to have $\tilde{\alpha}$ as large as possible. 

Based on the discussion above we here propose a definition of what it means that a family of approximations $\{u_h\}$
to an ill-posed problem of the form \eqref{eq:UC} is optimally convergent.
\begin{definition}\label{def:optimal}
Assume that $u\in H^{k+1}(\Omega),\, k\in \N$, solves the unique continuation problem \eqref{eq:UC}.
Let $\alpha\in(0,1)$ be the largest value for which the conditional stability estimate \eqref{eq:3sphere} holds.
Let $\{u_h\}_{h>0}$ be a family of functions in $H^1(\Omega)$.
If the family $\{u_h\}$ satisfies the inequality \eqref{eq:optimal_def} with $\frac{\alpha_1}{1+\alpha_1 - \alpha_2} = \alpha$, then we say that its convergence is optimal.
\end{definition}
\begin{remark}
An optimal $\alpha\in(0,1)$ in the stability estimate \eqref{eq:3sphere} is provided in Theorem \ref{thm:opt_alpha}.
We prove below that, independently of the method used, no family $\{u_h\} \subset H^1(\Omega)$ of approximations to the solution of \eqref{eq:UC} can satisfy \eqref{eq:optimal_def} with $\frac{\alpha_1}{1+\alpha_1 - \alpha_2} > \alpha$.
In particular, no method can exceed the convergence rate in \eqref{eq:optimal_conv} without increasing the sensitivity to data perturbations nor can it improve this sensitivity without decreasing the convergence rate, i.e. there exist no $\alpha_1, \alpha_2 \in [\alpha,1)$ with $\alpha_1 > \alpha$ or $\alpha_2 > \alpha$, such that 
\begin{align*}
	\norm{u-u_h}_{L^2(B)} 
	\lesssim 
	h^{\alpha_1 k} \norm{u}_{H^{k+1}(\Omega)} 
	+ h^{(\alpha_2 - 1) k} \norm{\delta q}_{L^2(\omega)}.
\end{align*}
\end{remark}

\begin{remark}	
The question of constructing such an optimal method with $\alpha_1 > \alpha$ is currently an open question.
\end{remark}

\subsection{Proof of optimality}\label{sec:optimal}
The following Caccioppoli-type inequality is known but we give a short proof for the convenience of the reader.
\begin{lemma}\label{lem:shiftnorm}
Let $r_3 < r_4 < R$ and $k \ge 0$. Then for all $w \in
H^{k+1}(B(R))$ satisfying $\Delta w = 0$ in $B(R)$ there holds
\[
\|w\|_{H^{k+1}(B(r_3))} \lesssim \|w\|_{L^2(B(r_4))}.
\]
\end{lemma}
\begin{proof}
Divide the interval $(r_3,r_4)$ in $k+1$ subintervals $(R_j,R_{j+1})$ of equal length with  
$R_j=r_3 + j \delta R$, $R_0=r_3$, $R_{k+1} = r_4, \delta R = R_{j+1} - R_j = (r_4 - r_3)/(k+1)$. For an index $j=0,\hdots,k$ 
choose a $\chi \in C_0^\infty(B(R_{j+1}))$ such  that 
$\chi = 1$ in $B(R_j)$ and write $v = \chi y$, where $y \in
H^1(B(R_{j+1}))$ and $\Delta y = 0$.
Then, if  $[\Delta, \chi]$ denotes the commutator $\Delta \chi - \chi \Delta$, 
\begin{equation*}
\left\{
\begin{aligned}
	\Delta v &= [\Delta, \chi] y &\mbox{in }& B(R_{j+1}) \\
	v &= 0 &\mbox{on }& \partial B(R_{j+1}).
\end{aligned}
\right.
\end{equation*}
and therefore by \eqref{eq:stab} we have that
\begin{equation}\label{eq:H1toL2}
\begin{aligned}
\norm{y}_{H^1(B(R_j))} 
&\le
\norm{v}_{H^1(B(R_{j+1}))}
\le
\norm{\Delta v}_{H^{-1}(B(R_{j+1}))}\\
&\lesssim \norm{[\Delta, \chi] y}_{H^{-1}(B(R_{j+1})}
\lesssim \norm{y}_{L^2(B(R_{j+1}))}.
\end{aligned}   
\end{equation}
Here in the last step we used that
\begin{align*}
([\Delta, \chi] y,w)_{B(R_{j+1})} &= ((\Delta \chi) y + 2 \nabla \chi \cdot \nabla y,
w)_{B(R_{j+1})}\\
&= - ((\Delta \chi) y,w)_{B(R_{j+1})}-2 (y, \nabla \chi \cdot \nabla w)_{B(R_{j+1})} \\
&\lesssim
\norm{y}_{L^2(B(R_{j+1}))} \|w\|_{H^{1}(B(R_{j+1}))}.
\end{align*}
Let $y= D^{k-j} w$ where $D^{k-j}$ denotes an arbitrary partial derivative of order $k-j$, $j = 0,\hdots,k$.
Then $\Delta y = 0$. It follows from equation \eqref{eq:H1toL2} that
\[
\norm{y}_{H^1(B(R_j))} \lesssim \norm{y}_{L^2(B(R_{j+1}))}.
\]
By applying this to all partial derivatives of order $k-j$, we see that
\[
\norm{w}_{H^{k+1-j}(B(R_j))} \lesssim \norm{w}_{H^{k-j}(B(R_{j+1}))}.
\]
Hence by applying this inequality sequentially for $j = 0,\hdots,k$ we see
that
\[
\norm{w}_{H^{k+1}(B(r_3))} = \norm{w}_{H^{k+1}(B(R_0))}
\lesssim \hdots \lesssim \norm{w}_{H^{1}(B(R_{k}))} \lesssim \norm{w}_{L^2(B(r_4))}.
\]
This concludes the proof.
\qed
\end{proof}

\begin{theorem}\label{thm:optimal}
Let $0<r_1<r_2<r_3<R$ and let $\omega = B(r_1)$, $B = B(r_2)$, $\Omega = B(r_3)$.
Let $\alpha \in (0,1)$ be the optimal exponent in Theorem \ref{thm:opt_alpha}.
Let $u\in H^{k+1}(B(R))$ satisfy $\Delta u = 0$ in $B(R)$ and let $q = u|_\omega$.
Consider a family of mappings $\{F_h\}_{h>0},\, F_h : L^2(\omega) \to H^1(\Omega),\, F_h(q + \delta q) =: u_h $, for all $\delta q \in L^2(\omega)$.
Then there exist no $\alpha_1, \alpha_2 \in (0,1)$ with $\frac{\alpha_1}{1+\alpha_1 - \alpha_2} > \alpha$, such that 
\begin{align}\label{eq:too_good}
\norm{u-u_h}_{L^2(B)} 
\lesssim 
h^{\alpha_1 k} \norm{u}_{H^{k+1}(\Omega)} 
+ h^{(\alpha_2 - 1) k} \norm{\delta q}_{L^2(\omega)}.
\end{align}
In particular, there exist no $\alpha_1, \alpha_2 \in [\alpha,1)$ with $\alpha_1 > \alpha$ or $\alpha_2 > \alpha$ such that \eqref{eq:too_good} holds.
\end{theorem}
\begin{proof}
We give a proof by contradiction.
Assume that there exist $\alpha_1, \alpha_2 \in (0,1)$ with
$$\tilde \alpha := \frac{\alpha_1}{1+\alpha_1 - \alpha_2} > \alpha$$
such that \eqref{eq:too_good} holds.
Taking $u=0$ and $\delta q = \tilde u|_\omega$ for $\tilde u$ satisfying $\Delta \tilde u = 0$ in $B(R)$, the estimate \eqref{eq:too_good} reduces to
\begin{align*}
\norm{F_h(\tilde u|_\omega)}_{L^2(B)} 
\lesssim 
h^{(\alpha_2 - 1) k} \norm{\tilde u}_{L^2(\omega)}.
\end{align*}
Using \eqref{eq:too_good} again with $u = \tilde u$ and $\delta q = 0$, we get 
\begin{align*}
\norm{\tilde u - F_h(\tilde u|_\omega)}_{L^2(B)} 
\lesssim 
h^{\alpha_1 k} \norm{\tilde u}_{H^{k+1}(\Omega)}.
\end{align*}
Hence 
\begin{equation}\label{eq:back_to_cont}
\begin{aligned}
\norm{\tilde u}_{L^2(B)}
&\le 
\norm{\tilde u - F_h(\tilde u|_\omega)}_{L^2(B)}
+ \norm{F_h(\tilde u|_\omega)}_{L^2(B)} \\
&\lesssim 
h^{\alpha_1 k } \norm{\tilde u}_{H^{k+1}(\Omega)}
+ h^{(\alpha_2 - 1) k} \norm{\tilde u}_{L^2(\omega)}.
\end{aligned}
\end{equation}    
We will write $u = \tilde u$ from now on, and recall that $u$ is an arbitrary solution to $\Delta u = 0$ in $B(R)$.
For a nonzero $u$, we define
\begin{align*}
	r := \frac{\norm{u}_{L^2(\omega)}}{\norm{ u}_{H^{k+1}(\Omega)}},
\end{align*}
and choose $h > 0$ such that 
\begin{align*}
	h^{\alpha_1 k} \norm{ u}_{H^{k+1}(\Omega)}
	= h^{(\alpha_2 - 1) k} \norm{ u}_{L^2(\omega)},
\end{align*}
that is,
\begin{align*}
	h = r^{1/((\alpha_1+1-\alpha_2)k)}.
\end{align*}
With this choice, inequality \eqref{eq:back_to_cont} reduces to
\begin{equation}\label{eq:three_ball2}
\norm{ u}_{L^2(B)} \lesssim \norm{ u}_{L^2(\omega)}^{\tilde \alpha} \norm{ u}_{H^{k+1}(\Omega)}^{1-\tilde \alpha},
\end{equation}
which trivially holds for the zero solution also.
Observe that (\ref{eq:three_ball2}) would right away contradict the optimality of $\alpha$ in Theorem \ref{thm:opt_alpha} if the $H^{k+1}$-norm on its right-hand side was an $L^2$-norm.
To weaken this norm, we can use Lemma \ref{lem:shiftnorm} to get that $\| u\|_{H^{k+1}(B(r_3))} \lesssim \| u\|_{L^2(B(r_4))}$ for $r_3 < r_4 < R$.
Hence, using this bound in \eqref{eq:three_ball2} leads to
\begin{align}\label{eq:three_ball3}
\norm{u}_{L^2(B)} \lesssim \norm{ u}_{L^2(\omega)}^{\tilde \alpha} \norm{ u}_{L^2(B(r_4))}^{1-\tilde \alpha}.
\end{align}
We now denote by $\hat{\alpha}$ the optimal exponent corresponding to $r_4$ in the three-ball estimate in Theorem \ref{thm:opt_alpha}, for which
\begin{equation}\label{eq:three_ball4}
	\norm{u}_{L^2(B)} \lesssim \norm{ u}_{L^2(\omega)}^{\hat{\alpha}} \norm{ u}_{L^2(B(r_4))}^{1-\hat{\alpha}},
\end{equation}
for any harmonic function $u$.
This means that such an inequality cannot hold with an exponent larger than $\hat{\alpha}$.
However, since $\hat{\alpha}$ depends continuously on $r_4$, by considering $r_4 > r_3$ sufficiently close to $r_3$ we can get $\hat{\alpha}$ arbitrarily close to $\alpha$, i.e. $\tilde{\alpha} > \hat{\alpha} > \alpha$.
Thus inequality \eqref{eq:three_ball3} holds with $\tilde{\alpha} > \hat{\alpha}$, which contradicts the optimality of $\hat{\alpha}$ in \eqref{eq:three_ball4}.

Let us finally show that if $\alpha_1, \alpha_2 \in [\alpha,1)$ with $\alpha_1 > \alpha$ or $\alpha_2 > \alpha$, then $\tilde \alpha > \alpha$.
Consider first the case $\alpha_1 < \alpha_2$.
As $\alpha_j \in [\alpha,1)$, $j=1,2$, there holds $\alpha_2 - \alpha_1 \in (0,1)$ and
\begin{align*}
	\tilde \alpha \ge \frac{\alpha}{1-(\alpha_2 - \alpha_1)}
	> \alpha.
\end{align*}
For the case $\alpha_1 \ge \alpha_2$, we have that $\alpha_1 = \alpha + \epsilon$ for some $\epsilon > 0$,
and    
\begin{align*}
	\tilde \alpha - \alpha \ge \frac{\alpha + \epsilon}{1 + \epsilon} - \alpha = \frac{(1-\alpha)\epsilon}{1 + \epsilon} > 0,
\end{align*}
which concludes the proof.
\qed
\end{proof}

\begin{remark}
	The proof of Theorem \ref{thm:optimal} is still valid if we assume \eqref{eq:too_good} to hold with a weaker norm instead of $\norm{u}_{H^{k+1}(\Omega)}$.
\end{remark}

\begin{remark}
	The approximation $u_h$ in Theorem \ref{thm:optimal} depends only on $h$ and $q+\delta q$.
	This result does not exclude the possibility of a regularisation method that uses more information, for example the size of the perturbation $\delta q$.
	The optimal method that we present in the following section can also use this information, see Remark \ref{remark:hmin}.
\end{remark}

\section{Primal-dual finite element methods with weakly consistent regularisation}\label{sec:FEM}
In this section we will use a finite element method with weakly consistent stabilisation to construct a sequence of approximate solutions for unique continuation \eqref{eq:UC} that satisfy the error estimate \eqref{eq:optimal_conv}, showing that the optimal convergence for this ill-posed problem can be attained by a discrete approximation method.
This discussion is based on ideas from \cite{Bu16,BHL18}, modified to match the assumptions of the theoretical developments above.

Let $\{\mathcal{T}\}$ be a quasi-uniform family of triangulations of $\Omega$, where triangles $T$ with curved boundaries are allowed so that the the covering of $\Omega$ is exact \cite{Zlamal70,Bern89}.
On these meshes we define a $C^0$ finite element space $V_h \subset H^1(\Omega)$, consisting of piecewise polynomials of order $k$ (after mapping of the triangles to a reference element).
We also let $V_{0h} = V_h \cap H^1_0(\Omega)$.
It then follows that there exist interpolants $\Pi_h:H^{1}(\Omega) \mapsto V_h$ \cite[Corollary 4.1]{Bern89} and
$\Pi_h^0: H^1_0(\Omega) \mapsto V_{0h}$ \cite[Corollary 5.2]{Bern89} for which the following interpolation estimates hold
\begin{equation}\label{eq:approx}
\|u - \Pi_h u\|_{T} + h \|\nabla (u - \Pi_h u)\|_T  + h^2 \|D^2(u - \Pi_h u)\|_{T} \lesssim
h^{k+1} |u|_{H^{k+1}(\Delta_T)},
\end{equation}
where $\Delta_T := \{T' \in \mathcal{T}: T' \cap T \ne \emptyset \}$ and $D^2 u$ is the Hessian of $u$, and
\begin{equation}\label{eq:low_order_approx}
\|w - \Pi^0_h w\|_{L^2(\Omega)} + h \|\Pi_h^0 w\|_V \lesssim
h \|w\|_V.
\end{equation} 
We will also use the broken norm defined by
\[
\|v\|_{\mathcal{T}} := \left( \sum_{T \in \mathcal{T}} \|v\|_T^2\right)^{1/2}.
\]
To set up the numerical method, we formulate the continuation problem \eqref{eq:UC} as pde-constrained optimization and consider the Lagrangian $L_h : V_h \times V_{0h} \to \R$,
\begin{equation*}
	L_h(u_h,z_h) := \underbrace{{\textstyle\frac12} \norm{u_h - \tilde{q}}^2_{L^2(\omega)}}_{\text{data fit}} + \underbrace{a(u_h,z_h)}_{\text{pde constraint}} + \underbrace{{\textstyle\frac{1}{2}} s(u_h,u_h) - {\textstyle\frac{1}{2}} a(z_h,z_h)}_{\text{discrete regularisation}}.
\end{equation*}
By taking its saddle points, we define the finite element method as follows: find $(u_h,z_h)  \in V_h \times V_{0h}$ such that
\begin{equation}
\begin{aligned}
a(u_h,w_h) - a(z_h,w_h) &= 0 \\
a(v_h,z_h) + s(u_h,v_h) + (u_h,v_h)_{L^2(\omega)} &=  (\tilde q,v_h)_{L^2(\omega)} 
\end{aligned}    
\end{equation}
for all $(v_h,w_h)  \in V_h \times V_{0h}$, with
\begin{equation}\label{eq:stab_def}
\begin{aligned}
	s(u_h,v_h):={}& \sum_{T \in \mathcal{T}}
	\left( (h_T^2 \Delta u_h, \Delta v_h)_T +
	\sum_{F \in \partial T} (h_T  \jumpF{\nabla u_h}, \jumpF{\nabla v_h})_{F\setminus \partial \Omega} \right)\\
	&+h^{2k} (u_h,v_h)_{L^2(\Omega)} 
\end{aligned}
\end{equation}
where $F$ denotes a face of a triangle $T$ and the jump of the gradient over a face $F$ is defined by
$\jumpF{\nabla u_h} := \nabla u_h\vert_{T_1} n_{T_1} + \nabla u_h\vert_{T_2} n_{T_2}$
for $F = \bar T_1 \cap \bar T_2$, with $n_{T_m}$ the outward pointing unit normal of the triangle $T_m$.
For a more compact formulation we introduce the global form $A_h$,
\begin{align*}
	A_h[(x_h,y_h),(v_h,w_h)] :={}& a(x_h,w_h) - a(y_h,w_h)\\
	& +a(v_h,y_h) + s(x_h,v_h) + (x_h,v_h)_{L^2(\omega)}
\end{align*}
to write: find $(u_h,z_h)  \in V_h \times V_{0h}$ such that
\begin{equation}\label{eq:FEM}
A_h[(u_h,z_h),(v_h,w_h)] = (\tilde q,v_h)_{L^2(\omega)} 
\end{equation}
for all $(v_h,w_h)  \in V_h \times V_{0h}$.
Observe that this form satisfies the consistency property
\begin{equation}\label{eq:gal_ortho}
A_h[( u - u_h, -z_h),(v_h,w_h)] = h^{2k} (u,v_h) - (\delta
q, v_h)_{L^2(\omega)}.
\end{equation}
To show that this method satisfies the error bound \eqref{eq:optimal_conv},
we only need to verify that it satisfies \eqref{eq:res}, \eqref{eq:data} and \eqref{eq:apriori} (which represent the design criteria for the method).
To this end we introduce the norm 
$$
\triple{(v_h,w_h)}^2 := s(v_h,v_h)+ \| w_h\|^2_V+ \|v_h\|^2_{L^2(\omega)}
$$
and we observe that the formulation satisfies the positivity property
\begin{equation}\label{eq:positivity}
\triple{(u_h,z_h)}^2 =  A_h[(u_h,z_h),(u_h,-z_h)] 
\end{equation}
which ensures the existence of a discrete solution $(u_h, z_h) \in V_h
\times V_{0h}$ for all $h>0$.

We proceed by first proving convergence in the $S$-norm, which immediately gives \eqref{eq:data}.
The proof of the other two bounds and satisfaction of \eqref{eq:optimal_conv} then follow as a corollary.
First we establish an approximation result for the $S$ norm.
\begin{lemma}\label{lem:stab_approx}
Let $v \in H^{k+1}(\Omega)$, then there holds
\[
\triple{(v - \Pi_h v,0)} \lesssim h^{k} \|v\|_{H^{k+1}(\Omega)}.
\]
\end{lemma}
\begin{proof}
By the definition of the $S$-norm we see that
\begin{equation*}
	\begin{aligned}
		\triple{(v - \Pi_h v,0)}^2 ={}& \|h \Delta (v - \Pi_h v)\|^2_{\mathcal{T}} +
		\sum_{T \in \mathcal{T}} \sum_{F \in \partial T\setminus \partial \Omega} h_T \|\jumpF{\nabla \Pi_h v}\|^2_F \\
		&+ h^{2k} \|v - \Pi_h v\|_{L^2(\Omega)}^2 + \|v - \Pi_h v\|^2_{L^2(\omega)}.
	\end{aligned}
\end{equation*}
By the approximation property \eqref{eq:approx} we have that
\[
\|h \Delta (v - \Pi_h v)\|^2_{\mathcal{T}}  + h^{2k} \|v - \Pi_h
v\|_{L^2(\Omega)}^2+ \|v - \Pi_h v\|^2_{L^2(\omega)} \lesssim h^{2k} |v|_{H^{k+1}(\Omega)}.
\]
For the term measuring the jump of $\Pi_h v$ over element faces we note that
\begin{equation*}
	\begin{aligned}
		\|\jumpF{\nabla \Pi_h v}\|^2_F &\leq \|\jumpF{\nabla (v - \Pi_h v)}\|^2_F \\
		&\leq h^{-1} \|\nabla (v - \Pi_h v)\|^2_{T_1 \cup T_2} + h \|D^2 (v - \Pi_h v)\|^2_{T_1 \cup T_2}
	\end{aligned}
\end{equation*}
where we used the regularity of $v$ and the trace inequality \cite{MS99}
\[
\|v\|_{\partial T} \lesssim h^{-\frac12} \|v\|_T + h^{\frac12}
\|\nabla v\|_T, \quad \forall v \in H^1(T).
\]
We conclude by applying \eqref{eq:approx} once again and summing over all the faces.
\qed
\end{proof}
\begin{proposition}\label{prop:res_conv}
Let $(u_h,z_h)$ denote the solution to \eqref{eq:FEM} and let $u \in H^{k+1}(\Omega)$ be the solution to \eqref{eq:UC}, then there holds
\[
\triple{(u - u_h,z_h)} \lesssim h^{k} \|u\|_{H^{k+1}(\Omega)} +
\|\delta q\|_{L^2(\omega)}.
\]
\end{proposition}
\begin{proof}
First we decompose the error $u -u_h = u - \Pi_h u + \underbrace{\Pi_h u - u_h}_{=:e_h}$ in the continuous and discrete parts.
By the triangle inequality and Lemma \ref{lem:stab_approx} it is enough to bound $\triple{(e_h,z_h)}$.
Using \eqref{eq:positivity} and \eqref{eq:gal_ortho} we have
\begin{equation*}
\begin{aligned}
	\triple{(e_h, z_h)}^2 = A_h[(e_h,-z_h),(e_h,z_h)] ={}& A_h[(\Pi_h u - u,0),	(e_h,z_h)]\\
	&- (\delta q, e_h)_{L^2(\omega)} + h^{2k} (u, e_h)_{L^2(\Omega)}.
\end{aligned}
\end{equation*}
For the last two terms on the right hand side we have
\[
-(\delta q, e_h)_{L^2(\omega)} + h^{2k} (u, e_h)_{L^2(\Omega)} \leq
(\|\delta q\|_{L^2(\omega)} + h^{k} \|u\|_{L^2(\Omega)}) \triple{(e_h, 0)}.
\]
Finally, the following continuity holds
\begin{equation}\label{eq:cont1}
A_h[(\Pi_h u - u,0),
(e_h,z_h)] \lesssim \| \Pi_h u - u\|_* \triple{(e_h, z_h)}
\end{equation}
where 
\[
 \|v \|_*:= \| v\|_V +  \triple{v}.
\]
To prove the continuity \eqref{eq:cont1}  recall that by definition
\[
A_h[(\Pi_h u - u,0),
(e_h,z_h)]  = a(\Pi_h u - u,z_h) + s(\Pi_h u - u,e_h) + (\Pi_h u - u,e_h)_{L^2(\omega)}.
\]
Using the Cauchy-Schwarz inequality we have that
\[
a(\Pi_h u - u,z_h) \leq  \| \Pi_h u - u\|_V \triple{(0, z_h)}
\]
and
\[
s(\Pi_h u - u,e_h) + (\Pi_h u - u,e_h)_{L^2(\omega)} \leq \triple{(\Pi_h u - u,0)} \triple{(e_h,0)}.
\]
We end the proof by observing that by equation \eqref{eq:approx} and Lemma \ref{lem:stab_approx} there holds
\begin{equation*}\label{eq:star_approx}
\| \Pi_h u - u \|_* \lesssim h^{k} |u|_{H^{k+1}(\Omega)}.
\end{equation*}
\qed
\end{proof}
\begin{corollary}\label{cor:fem_error_estimate}
Under the same hypothesis as for Proposition \ref{prop:res_conv} there holds
\[
\|u_h\|_{L^2(\Omega)} \lesssim \|u\|_{H^{k+1}(\Omega)} +
h^{-k}\|\delta q\|_{L^2(\omega)}
\]
and
\[
\|\Delta u_h\|_{H^{-1}(\Omega)} \lesssim h^{k} \|u\|_{H^{k+1}(\Omega)} +
\|\delta q\|_{L^2(\omega)}.
\]
Finally $u_h$ satisfies the error bound \eqref{eq:optimal_conv}.
\end{corollary}
\begin{proof}
First we observe that the third claim is an immediate consequence of the first two and Proposition \ref{prop:res_conv}.
Indeed, this follows from the discussion of Section \ref{sec:ill_posed_opt}, using the error bound \eqref{eq:protobound} and equations \eqref{eq:res} - \eqref{eq:apriori}.

The first inequality is immediate by Proposition \ref{prop:res_conv} observing that
\[
\|u_h\|_{L^2(\Omega)} \leq \| u-u_h\|_{L^2(\Omega)}+\|u\|_{L^2(\Omega)},
\]
and, for the first term in the right hand side,
\[
\| u-u_h\|_{L^2(\Omega)} \leq h^{-k} \triple{(u-u_h,0)} \lesssim \|u\|_{H^{k+1}(\Omega)} +
h^{-k}\|\delta q\|_{L^2(\omega)}.
\]
For the second inequality, by definition
\[
\|\Delta u_h\|_{H^{-1}(\Omega)} = \sup_{w \in V_0\setminus \{0\}} \frac{a(u_h,w)}{\|w\|_V}.
\]
Using \eqref{eq:gal_ortho}, followed by integration by parts, we see that for all $w_h \in V_{0h}$
\begin{multline*}
a(u_h,w) = a(u_h-u ,w) = a(u_h - u,w - w_h) + a(z_h,w_h) \\=\sum_{T
  \in \mathcal{T}} \left(-(\Delta (u_h - u),w - w_h)_{L^2(T)} + (\jumpF{\nabla
    u_h},w - w_h)_{L^2(\partial T \setminus \partial \Omega)} \right)   + a(z_h,w_h).
\end{multline*}
Choosing $w_h = \Pi^0_h w$ and using the Cauchy-Schwarz inequality in the first term of the right hand side
and the continuity of $a$ in the second, followed by \eqref{eq:low_order_approx}, we see that
\begin{align*}
{}&\sum_{T\in \mathcal{T}} \left( -(\Delta (u_h - u),w - w_h)_{L^2(T)}
+ (\jumpF{\nabla u_h},w - w_h)_{L^2(\partial T \setminus \partial \Omega)} \right)
+ a(z_h,w_h) \\
&\lesssim \triple{(u_h - u,z_h)} \|w\|_V.	
\end{align*}
The conclusion now follows using Proposition \ref{prop:res_conv} to obtain the desired bound
\[
\|\Delta u_h\|_{H^{-1}(\Omega)} = \sup_{w \in V_0\setminus \{0\}}
\frac{a(u_h,w)}{\|w\|_V} \lesssim h^{k} \|u\|_{H^{k+1}(\Omega)} +
\|\delta q\|_{L^2(\omega)}.
\]
\qed
\end{proof}

\begin{remark}\label{remark:hmin}
Both for the well-posed problem \eqref{eq:Poisson} and the ill-posed problem \eqref{eq:UC} there is a lower bound for how  well the exact solution can be approximated if the data are perturbed.
In the well-posed case the limit is trivially given by $\|\delta f\|_{{V_0'}}$ in \eqref{eq:error_bound_highorder}, whereas in the ill-posed case the lower bound occurs when the approximation error term and the perturbation term are equal in \eqref{eq:optimal_conv}, that is, 
\begin{equation*}
h^{\alpha k}
\|u\|_{H^{k+1}(\Omega)} = h^{(\alpha -1)k}
\|\delta q\|_{L^2(\omega)}.   
\end{equation*}
This gives a theoretical lower bound on $h$ for which refining the mesh decreases the error bound,
\[
h_{min} = (\|\delta q\|_{L^2(\omega)}/ \|u\|_{H^{k+1}(\Omega)})^{1/k}.
\]
If $h_{min}$ is known the numerical scheme can be designed to stagnate at the level of the best approximation, by modifying the last term in the definition of the stabilisation \eqref{eq:stab_def} to read
$\max(h,h_{min})^{2k} (u_h,v_h)_{L^2(\Omega)} $. This shows the connection between this stabilising term and classical Tikhonov regularisation and similar tools as for the latter can be applied here to optimise the parameter compared to perturbations in data.
It is straighforward to show that this leads to stagnation at
\[
\|u - u_h\|_{L^2(B)} \lesssim  h_{min}^{\alpha k} \|u\|_{H^{k+1}(\Omega)}
= \|\delta q\|_{L^2(\omega)}^{\alpha} \|u\|_{H^{k+1}(\Omega)}^{1-\alpha}.
\]
Here the implicit constant may depend on $k$. A similar kind of bound was obtained in \cite[Theorem 2.2]{DMS23}.
\end{remark}

\section{Conclusion}
In this paper we have shown that the convergence order of the approximation error for unique continuation problems, obtained by combining the approximation orders of the data fitting and the pde-residual with the conditional stability, can not be improved without increasing the sensitivity to perturbations.
This shows that the asymptotic accuracy of the methods for unique continuation discussed in \cite{Bu14b,Bu16,BHL18,BLO18,BNO19,BNO20a,DMS23} is optimal, in the sense that it is impossible to design a method with better convergence properties.
The only remaining possibilities to enhance the accuracy of approximation methods is either to resort to adaptivity, or to introduce some additional a priori assumption to make the continuous problem more stable, such as finite dimensionality of target quantities (see \cite{burman2023finite}).

\begin{acknowledgements}
The authors thank Tuomo Kuusi for valuable advise on three-ball inequalities, and the two anonymous reviewers for their valuable comments and suggestions, which have improved the paper. 
\end{acknowledgements}

%
%

\bibliographystyle{spmpsci}      
\bibliography{biblio}   

\end{document}